\documentclass[12pt]{amsart}

\usepackage{amsmath,amssymb}
\usepackage[all]{xy}

\newcommand{\A}{\mathbb A}
\newcommand{\Z}{\mathbb Z}

\newcommand{\cO}{\mathcal O}
\newcommand{\cF}{\mathcal F}
\newcommand{\cG}{\mathcal G}

\newcommand{\bG}{\overline{G}}

\newcommand{\baf}{\bar f}

\newcommand{\bh}{\bar{h}}

\newcommand{\bp}{\bar{p}}

\newcommand{\m}{\mathfrak m}

\DeclareMathOperator{\lcm}{lcm} 
 
\DeclareMathOperator{\mult}{mult}

\DeclareMathOperator{\ord}{ord}
\DeclareMathOperator{\spec}{Spec}

\DeclareMathOperator{\lt}{LT}
\DeclareMathOperator{\lm}{LM}
\DeclareMathOperator{\lc}{LC}
\DeclareMathOperator{\nf}{NF}
\DeclareMathOperator{\snf}{S-NF}
\DeclareMathOperator{\spo}{Spoly}
\DeclareMathOperator{\mon}{Mon}
\DeclareMathOperator{\In}{In}
\DeclareMathOperator{\gr}{Gr}
\DeclareMathOperator{\ec}{ecart}

\renewcommand{\:}{\colon}

\newcommand{\gen}[1]{\langle #1 \rangle}
\newcommand{\ol}[1]{\overline #1 }

\newcommand{\defgen}[2]{\left\langle #1 \mid #2 \right\rangle}
\newcommand{\defset}[2]{{\left\{\left.#1\,\right| \,#2  \right\}}}

\newcommand{\ds}{>_{ds}}
\newcommand{\al}{\alpha}
\newcommand{\xa}{x^{\alpha}}

\newcommand{\dis}{\displaystyle}

 \theoremstyle{plain}
\newtheorem{thm}{Theorem}[section]

\newtheorem{prop}[thm]{Proposition}

 \theoremstyle{definition}
\newtheorem{defn}[thm]{Definition}

\newtheorem{ex}[thm]{Example}

\newtheorem*{acknowledgement}{Acknowledgement}

\theoremstyle{remark}

\newtheorem{rem}[thm]{Remark}

\numberwithin{equation}{section}

\newcommand{\thmref}[1]{Theorem~\ref{#1}}

\newcommand{\proref}[1]{Proposition~\ref{#1}}
\newcommand{\remref}[1]{Remark~\ref{#1}}

\newcommand{\defref}[1]{Definition~\ref{#1}}

\newcommand{\sref}[1]{Section~\ref{#1}}

\begin{document}

\title[The multiplicity of cyclic coverings of a singularity]{The multiplicity of cyclic coverings of a singularity of an algebraic variety}

\author{Toya Kumagai}
\address{Graduate School of Science and Engineering,Yamagata University, 
 Yamagata 990-8560, Japan.}
\email{s221269m@st.yamagata-u.ac.jp}

\author{Tomohiro Okuma}
\address{Department of Mathematical Sciences, 
Yamagata University, 
 Yamagata 990-8560, Japan.}
\email{okuma@sci.kj.yamagata-u.ac.jp}
\thanks{The second-named author was partially supported by JSPS Grant-in-Aid 
for Scientific Research (C) Grant Number 21K03215. }


\begin{abstract}
Let $V$ be an affine algebraic variety, and let  $p\in V$ be a singular point. 
For a regular function $g$ on $V$ such that $g(p)=0$ and for a positive integer $n$, we consider the cyclic covering $\phi_n\: V_n \to V$ of degree $n$ branched along the hypersurface defined by $g$.
We will prove that for sufficiently large $n$,  the tangent cone of $V_n$ at $\phi_n^{-1}(p)$ is, as an affine variety, the product of the tangent cone of the branch locus and the affine line.  
In particular, the multiplicity of the singularity $\phi_n^{-1}(p) \in V_n$, 
which is a function of $n$ determined by $V$ and $g$,  remains constant for sufficiently large $n$.
This result generalizes Tomaru's theorem for normal surface singularities.
\end{abstract}

\maketitle


\section{Introduction}

Let $V$ be an affine algebraic variety over an algebraically closed field $K$ and let $p\in V$ be a singular point. 
Then we call the pair $(V,p)$ a singularity.
We denote by $A$ the local ring $\cO_{V,p}$ of the variety $V$ at $p$, which is called the local ring of the singularity $(V,p)$.
The {\em multiplicity} $\mult(V,p)$ of the singularity $(V,p)$ is defined as the (Hilbert-Samuel) multiplicity of the local ring $A$ (see \cite[\S 14]{commring}).
The multiplicity is a fundamental invariant playing a crucial role in singularity theory.
Let $\m$ denote the maximal ideal of $A$, and let $d=\dim (V,p):=\dim A$.
If $\{f_1, \dots, f_d\}\subset A$ generates an $\m$-primary ideal, then it defines a local finite morphism $(V,p) \to (\A_K^d,o)$ (cf. \cite[Theorem 14.5]{commring}). From a geometric point of view, $\mult(V,p)$ can be interpreted as the minimum of the degree of such morphisms (cf. \cite[Exercises 6.3.6]{Singular}, \cite[Theorems 14.13, 14.14]{commring} and \cite[Ch. VIII, \S 10, Corollary 2]{Zariski-Samuel-II}).

In this paper, we investigate the multiplicity of singularities on branched cyclic coverings of a given singularity.
Let $g$ be a non-zero regular function on $V$ such that $g(p)=0$, and let $y$ be the coordinate function of the affine line $\A^1_K$.
For $n\in \Z_{>0}$, the natural projection $\phi_n\: V_n:=\{y^n=g\}\subset V\times \A^1_K \to V$ defines a cyclic covering of degree $n$ branched along the hypersurface $\{g=0\}\subset V$. Clearly, the fiber $\phi_n^{-1}(p)$ consists of the point $p_n:=p\times \{0\}\in V_n$. 
From the geometric perspective, we have the inequality $\mult(V_n, p_n) \le n\cdot \mult(V,p)$. 
We will show the following.

\begin{thm}\label{t:T}
There exists a positive integer $N$ such that the function $\mult(V_n,p_n)$ of $n$ is constant for $n > N$.
\end{thm}

\noindent
This result is motivated by Tomaru's result \cite{tomaru-Kag} which proves the theorem above for normal complex analytic singularities of dimension two.
Tomaru used the fact that if $M$ represents the maximal ideal cycle on a resolution with nice property, then the multiplicity coincides with $-M^2$. 
To prove the theorem, he constructed a branched cyclic covering $\psi_n\: \tilde V_n \to \tilde V$ of resolution spaces of $V_n$ and $V$, and computed the maximal ideal cycle on $\tilde V_n$ from data of $\tilde V$ via $\psi_n$. 
While this method is effective  for analyzing the maximal ideal cycles, it becomes somewhat intricate when we focus just on the function $\mult(V_n,p_n)$.
In this paper, we adopt the standard bases of ideals of the localization of the polynomial rings, and prove a stronger result as follows (see \thmref{t:main} and \remref{r:TC}).

\

\begin{thm}\label{t:M}
There exists a positive integer $N$, which is determined by a standard basis of an ideal defining the singularity $(V,p)$ and the regular function $g$, such that the affine tangent cone of $(V_n, p_n)$ is isomorphic to the product of the affine tangent cone of the subscheme of $V$ defined by $g$ at $p$ and the affine line $\A^1_K$.
\end{thm} 

For these results, we do not make assumptions about the Cohen-Macaulayness or normality of the singularity.
Since the multiplicity is determined by the tangent cone, \thmref{t:M} implies \thmref{t:T}. The main theorem may also reveal which invariants of singularities are not determined (or bounded) by the invariants of their tangent cones (see \remref{r:pg}).  

In \sref{s:std}, we provide a brief overview of the algorithm and criterion for the standard basis and the description of the tangent cones, both of which are crucial for the proof of the main theorem in \sref{s:main}.

\begin{acknowledgement}
The authors are very grateful to the referee for a careful reading of the article and valuable comments which helped improve the description of terminologies and methods, as well as the proof of the main theorem.
\end{acknowledgement}

\section{Standard bases and tangent cones }\label{s:std}

In this section, we introduce some notation and provide a brief description of the standard bases, as well as the tangent cone of a local ring in terms of standard bases.
A standard basis is called a Gr{\" o}bner basis if the monomial ordering is global.
We refer to \cite{Singular} for basics about the standard basis and related results.
In the following, we denote an ideal generated by $F=\{f_1, \dots, f_m\}$ as $\gen{F}$ or $\gen{f_1, \dots, f_m}$.

Let $K$ be an algebraically closed field and let $K[x] = K[x_1, \dots, x_r]$ denote the  polynomial ring in $r$ variables over $K$.
Let $R=K[x]_{\gen{x}}$ denote the localization of $K[x]$ with respect to the maximal ideal $\gen{x}=\gen{x_1, \dots, x_r}$. 
We may regard $R$ as a subring of the power series ring $K[[x]]$ in a natural way.
A monomial $\prod_{i=1}^rx_i^{\alpha_i} \in K[x]$ is denoted by $x^{\alpha}$, where $\alpha=(\alpha_1, \dots, \alpha_r)\in (\Z_{\ge 0})^r$, 
and the degree of $x^{\alpha}$ is defined to be $\deg x^{\al}:=\sum_{i=1}^r\alpha_i$.
Let $\mon(x)=\mon(x_1, \dots, x_r)$ denote the semigroup of monomials,  namely, $\mon(x)=\defset{x^{\alpha}}{\alpha\in (\Z_{\ge 0})^r}$.
A total ordering $>$ on $\mon(x)$ is called a  {\em monomial ordering} if $x^{\alpha}>x^{\beta}$ implies $x^{\alpha}x^{\gamma}>x^{\beta}x^{\gamma}$ for all  $x^{\alpha}, x^{\beta}, x^{\gamma} \in \mon(x)$. A monomial ordering $>$ is said to be {\em global} (resp. {\em local})
if $x^{\alpha} > 1$ (resp. $x^{\alpha}<1$) for all $x^{\alpha}\in \mon(x)\setminus \{1\}$.
If $ \deg x^{\alpha} < \deg x^{\beta}$ implies $x^{\alpha} > x^{\beta}$, the ordering $>$ is called a {\em local degree ordering}. 
Local orderings play a crucial role in the analysis of ideals in the local ring $R$.
As an example, we introduce the {\em negative degree reverse lexicographical ordering} $\ds$, which is defined as follows:
\begin{align*}
x^{\alpha}\ds x^{\beta}  \overset{\text{def}}{\Longleftrightarrow}
& \deg x^{\alpha} < \deg x^{\beta}, \quad \text{or} \\
&  \deg x^{\alpha} = \deg x^{\beta} \ \text{and there exists $i$ such that} \\
& \quad \alpha_n=\beta_n, \dots,   \alpha_{i+1}=\beta_{i+1}, \alpha_i < \beta_i.
\end{align*}

In the following, we fix a local degree ordering  $>$  on $\mon(x)$.

Let $R^*$ denote the set of units of $R$. 
Let $f\in R\setminus \{0\}$.
Then there is $u\in R^* \cap K[x]$ with $u(0)=1$ such that $uf\in K[x]$. 
Let us write it as
\[
uf= \sum_{\al}a_{\al}\xa 
= \sum_{i=1}^na_{i}x^{\al(i)}, \quad a_{i}\ne 0 \ \ \text{for $1\le i\le n$}, \quad x^{\al(1)} > x^{\al(2)}>\dots
\]
Then we define:
\begin{itemize}
\item $\lm(f):=x^{\al(1)}$, the {\em leading monomial} of $f$,
\item $\lt(f):=a_1x^{\al(1)}$, the {\em leading term} of $f$, 
\item $\lc(f):=a_1$, the {\em leading coefficient} of $f$, 
\item $\ord(f):=\min\defset{\deg \xa}{a_{\al}\ne 0}=\deg x^{\al(1)}$,  the {\em order} of $f$,
\item $\dis \In (f):=\sum_{\deg \xa = \ord(f)} a_{\al}\xa$, the {\em initial form} of $f$.
Let $\In(0)=0$.
\end{itemize}

Obviously,  for $f\in K[x]\setminus \{0\}$, $\lm(f)$ appears in $\In(f)$.

\begin{defn}
Let $f, g \in R\setminus\{0\}$, and write $\lm(f)=x^{\al}$ and $\lm(g)=x^{\beta}$.
Let 
\[
\gamma=\lcm(\al,\beta):=(\max(\al_1, \beta_1), \dots, \max(\al_r,\beta_r)).
\]
Then the {\em $s$-polynomial} of $f$ and $g$ is defined to be
\[
\spo(f,g)=x^{\gamma-\al}f-\frac{\lc(f)}{\lc(g)}x^{\gamma-\beta}g.
\]
Clearly, $\spo(f,g)\in K[x]$ if $f,g\in K[x]$.
\end{defn}

\begin{defn}
Let $I\subset R$ be an ideal and $G\subset R$ a subset.
\begin{enumerate}
\item  The ideal 
\[
L(G):=\defgen{\lm(g)}{g\in G\setminus\{0\}} \subset K[x]
\]
is called the {\em leading ideal} of $G$.
\item If $G$ is a finite subset of $I$ and $L(I)=L(G)$, then $G$ is called a {\em standard basis} of $I$.

\item The {\em initial ideal} of $I$ is defined to be 
\[
\In (I) = \defgen{\In(f)}{f\in I \setminus \{0\}} \subset K[x].
\]
\end{enumerate}
\end{defn}

\begin{prop}[{cf. \cite[5.5.11--5.5.12]{Singular}}]
\label{p:In}
Let $I \subset \gen{x}\subset K[x]$ be an ideal, and let $A=R/IR$ and $\m \subset A$ the maximal ideal of $A$.
\begin{enumerate}
\item 
If $\{f_1, \dots, f_s\}$ is a standard basis of $I$, then $\In(I)=\gen{\In(f_1), \dots, \In(f_s)}$.
\item  The tangent cone $\gr_{\m}(A)=\bigoplus_{n\ge 0}\m^n/\m^{n+1}$ is isomorphic to $K[x]/\In(I)$ as graded $K$-algebra.
\end{enumerate}
In particular, the  multiplicity of the local ring $A$ is determined by the initial forms of a standard basis of $I$.
\end{prop}

We provide a very brief overview of the criterion for the standard bases and an algorithms to compute them with respect to a local ordering. These are essentially Buchberger's algorithms but have been generalized by Mora (\cite{Mora}).

\begin{defn}[{cf. \cite[1.6.4--1.6.5]{Singular}}]
Let $\cG$ denote the set of all finite subsets of $R$. 
A map 
\[
\nf\: R \times \cG \to R, \quad (f, G)\mapsto \nf(f\mid G),
\]
is called a {\em weak normal form} on $R$ if  the following are satisfied:
\begin{enumerate}
\item For all $G\in \cG$, $\nf(0\mid G)=0$.
\item For all $f\in R$ and $G\in \cG$, $\nf(f\mid G)\ne 0$ implies $\lm(\nf(f\mid G))\not \in L(G)$.
\item For all $f\in R$ and $G=\{g_1, \dots, g_s\}\in \cG$, there exists a unit $u\in R^{*}$ such that  $uf-\nf( f \mid G)$ (or, by abuse of notation, $uf$) has a standard representation with respect to $\nf( - \mid G)$, namely,
\[
uf-\nf(f\mid G) = \sum_{i=1}^s a_ig_i \quad (a_i \in R)
\]
with the property that $\lm(\sum_{i=1}^s a_ig_i) \ge \lm(a_ig_i)$ for all $i$ such that $a_ig_i\ne 0$.
\end{enumerate}

A weak normal form $\nf$ is called {\em polynomial} if, whenever $f\in K[x]$ and $G\subset K[x]$,  in the condition (3) above,  $u$ and $a_i$ can be chosen as elements of $K[x]$.
\end{defn}

\begin{defn}[cf.  {\cite[2.6]{Mora}}]
\label{d:snf}
Let $\cG$ denote the set of all finite subsets of $R$. 
Let $h\in R$ and $G\in \cG$. Assume that there exist sequences $\{h_0=h, h_1, \dots, h_m\}\subset R$ and $\{G_0=G, G_1, \dots, G_m\}\subset \cG$ such that  
\begin{gather*}
\exists g_i\in G_i, \ 
\lm(g_i) \mid \lm(h_i), \
h_{i+1}=\spo(h_i, g_i), \
G_{i+1}=G_i\cup \{h_i\}, \\
h_m=0 \text{ \ or \ $\lm(g) \nmid \lm(h_m)$ for all $g\in G_m$.}
\end{gather*}
Then we call $h_m$ an  {\em $s$-normal form} of $h$ with respect to $G$.
Clearly, we have  $\{h_i\}\subset K[x]$ if $h\in K[x]$ and $G\subset K[x]$. 
We write as 
\[
\snf(h\mid G)=0
\]
 to mean that an $s$-normal form of $h$ with respect to $G$ exists and is zero.
\end{defn}

Note that in general, a sequence $\{h_0=h, h_1, \dots, h_m\}\subset R$ with the property in \defref{d:snf} does not exist; in other words, the procedure may not terminate (cf. \cite[\S 1.7]{Singular}).

Buchberger's criterion for Gr{\" o}bner bases is also effective for standard bases with local ordering.

\begin{thm}
[Buchberger's Criterion (cf. {\cite[1.7.3]{Singular}})]
\label{t:Buchberger}
Let $I \subset R$ be an ideal and $G =\{g_1, \dots, g_s\} \subset I$. 
Let  $\nf$ be a weak normal form on $R$ with respect to $G$. 
Then the following are equivalent:
\begin{enumerate}
\item $G$ is a standard basis of $I$.
\item  $\nf( f \mid G) = 0$ for all $f\in I$.
\item Each $f\in I$ has a standard representation with respect to $\nf(- \mid G)$.
\item $G$ generates $I$ and $\nf(\spo(g_i,g_j) \mid G)=0$ for $i, j=1, \dots, s$.
\item $G$ generates $I$ and $\nf(\spo(g_i,g_j) \mid G_{ij})=0$ for 
a suitable subset $G_{ij}\subset G$ and  $i, j=1, \dots, s$.
\end{enumerate}
\end{thm}

The following proposition is also useful in the proof of our main theorem.

\begin{prop}
[See {\cite[2.5-2.6, 3.1]{Mora}}] \label{p:snfstd}
Let $I \subset R$ be an ideal and $G \subset I$ a finite subset generating $I$. 
Then $G$ is a standard basis of $I$ if $\snf(\spo(g,g') \mid G)=0$ for any $g, g'\in G$.
\end{prop}

There is an algorithm, proposed by Mora (\cite{Mora}, cf. \cite[1.7.6]{Singular}), that yields a polynomial weak normal form.
When employing a global ordering, $\nf(f \mid G)$ is obtained as the last element of a sequence $h_0=f, h_1, \dots, h_m$ such that $h_{i+1}=\spo(h_{i},h_i')$, where $h_i'$ is a suitable element of $G$. Mora's normal form with respect to a local ordering is obtained in a similar manner, but $h_i'$ is an element of $G\cup\{h_0, \dots, h_{i-1}\}$.
For the convenience of the readers, we state the algorithm.

For $f=\sum_{\al}a_{\al}\xa \in K[x]\setminus\{0\}$, let
\[
\ec(f) = \max\defset{\deg \xa}{a_{\al}\ne 0} -
\min\defset{\deg \xa}{a_{\al}\ne 0}.
\]

\begin{prop}
[cf. {\cite[1.7.6]{Singular}}]
\label{p:mora}
Let $\cG$ denote the set of all finite subsets of $K[x]$. 
Let $f\in K[x]\setminus\{0\}$ and  $G \in \cG$.
We construct sequences $\{h_0=f, h_1, \dots\}\subset K[x]$ and  $\{T_0 = G, T_1, \dots\}\subset \cG$ as follows: 
Suppose that we obtain $\{h_0, \dots, h_i\}$ and $\{T_0, \dots, T_i\}$.
If $h_i=0$ or $T'_i:=\defset{g\in T_i}{\lm(g) \mid \lm(h_i)}=\emptyset$, then we stop here.
In case $h_i\ne 0$ and $T'_i\ne \emptyset$, take $g\in T'_i$ with $\ec(g)=\min\defset{\ec(g')}{g'\in T'_i}$, let $T_{i+1}=T_i\cup\{h_i\}$ if $\ec(g)>\ec(h_i)$ and  $T_{i+1}=T_i$ otherwise, and let $h_{i+1}=\spo(h_i,g)$.
This algorithm terminates after a finite number of steps, and the last one in the sequence $\{h_0=f, h_1, \dots\}$ is a polynomial weak normal form of $f$ with respect to $G$.
\end{prop}

Note that Mora's normal form is also an $s$-normal form. The difference between them is whether the invariant $\ec$ is used in the process to obtain the normal forms.

Using Mora's  normal form, we can compute the standard basis  as follows (see \cite[1.7.1, 1.7.8]{Singular}).

\begin{prop}\label{p:std}
Let $I\subset R$ be an ideal generated by a finite subset $G_0\subset K[x]$ and let $\nf$ denote Mora's  polynomial weak normal form.
We construct a sequence $\{G_0 \subsetneq G_1 \subsetneq \dots\}$ of finite subsets of $K[x]$ as follows: if we obtain $\{G_0, \dots, G_i\}$ and 
if there exists $f,g\in G_i$ such that $h:=\nf(\spo(f,g)\mid G_i)\ne0$, then define $G_{i+1}$ to be $G_i\cup\{ h\}$.
This algorithm terminates after a finite number of steps, and the last subset in the sequence is a standard basis for $I$. 
\end{prop}

\section{The main results}\label{s:main}

In this section, we will prove the main theorem.
We basically use the notation from the preceding section.
Let $I \subset \gen{x}\subset K[x]$  be a prime ideal and $V\subset \A^r_K$ the affine variety defined by $I$. 
Let $A$ denote the local ring of the variety $V$ at the origin $o\in V$;
 namely, $A=R/IR$.
Let $K[x,y]=K[x_1, \dots, x_r, y]$ denote the polynomial ring in $r+1$ variables and let $B=K[x,y]_{\gen{x,y}}$, where $\gen{x,y}=\gen{x_1, \dots, x_r,y}$. 
Let $g\in \gen{x}\setminus I$.
For $n \in \Z_{>0}$, let $I_n=I+\gen{g-y^n}\subset K[x,y]$ and $B_n= B/ I_n B$.
Then $B_n$ is the local ring of the singularity $(V_n,p_n)$ in the Introduction, and  \thmref{t:T} immediately follows from  \proref{p:In} and the following.

\begin{thm}\label{t:main}
Let $\{f_1, \dots, f_m\} \subset K[x]\setminus \{0\}$  be a standard basis of $I$ and let $G_0=\{f_1, \dots, f_m, g\}$.
Let $\{g_1, \dots, g_t\}$ be a standard basis of $I+\gen{g}\subset K[x]$ which is obtained as the last element $G_k$ of the sequence  $G_0 \subsetneq \dots \subsetneq G_k\subset K[x]$, as described in \proref{p:std}.
Let $N=\max\defset{\ord(g_i)}{i=1, \dots, t}$.
Then, for $n>N$,  the tangent cone of the local ring $B_n$ is isomorphic to $K[x,y]/\gen{\In(g_1), \dots, \In(g_t)}$.
\end{thm}

\begin{proof}
Assume that $n>N$.
By \proref{p:In}, it is sufficient to prove that the ideal $I_n$ has a standard basis $\cG$ such that $\defset{\In(f)}{f\in \cG}= \{\In(g_1), \dots, \In(g_t)\}$.

Note that all polynomials in $G_k$ and those appearing in the algorithm for obtaining $G_k$ from $G_0$ are represented as linear combinations of polynomials in  $G_0$ with coefficients in $K[x]$. 
For each $f$ of those polynomials, we fix such a representation $f=\sum_{i=1}^m a_if_i +a_{m+1}g$, and define $\baf\in K[x,y]$ as the polynomial obtained from $f$ by replacing $g$ with $g-y^n$ in the representation; namely, $\baf= \sum_{i=1}^m a_if_i +a_{m+1}(g-y^n)$.
Let us express these procedures more precisely.
For $0 \le i <k$, we may assume that $h^{(i)}\in G_{i+1}\setminus G_i$ is a unique element and $h^{(i)}=\nf(\spo(p_1^{(i)},p_2^{(i)}) \mid G_i)$ with $p_1^{(i)}, p_2^{(i)}\in G_i$.
Let $\{h_0^{(i)}=\spo(p_1^{(i)},p_2^{(i)}), h_1^{(i)}, \dots, h_{\ell_i}^{(i)}=h^{(i)}\}$ be a sequence obtained by the algorithm in \proref{p:mora}. 
Let 
\[
\cF=\left( \bigcup_{i=0}^k G_i \right) 
\cup \left( \bigcup_{i=0}^{k-1} \{h_0^{(i)}, h_1^{(i)}, \dots, h_{\ell_i}^{(i)}\} \right)
\]
and let $K[x][s]:=K[x][s_1, \dots, s_{m+1}]$ be the polynomial ring in $m+1$ variables with coefficient ring $K[x]$.
We fix a map $\Phi\: \cF \to K[x][s]$ such that $\Phi(f)=\Phi(f)(s_1, \dots, s_{m+1})$ is of the form 
$\sum_{i=1}^{m+1} a_{i}(f)s_i$, where $a_i(f)\in K[x]$,  and that 
\[
f=\Phi(f)(f_1, \dots, f_m,g)=\sum_{i=1}^m a_i(f)f_i +a_{m+1}(f)g \in K[x].
\]
Moreover assume that $\Phi(f_i)=s_i$ for $i=1, \dots, m$, 
and $\Phi(g)=s_{m+1}$.
For each  $f\in \cF$, we define a polynomial $\ol{f}$ by
\[
\ol{f}=\Phi(f)(f_1, \dots, f_m,g-y^n)\in K[x,y].
\]
Let $\bG_i = \defset{\baf}{f\in G_i}\subset K[x,y]$.
Then $I_n=\gen{\bG_0}$.

We consider the following condition for $0\le j \le k$:
\begin{center}
$C(j)$: 
$\In (f) \in K[x]$ for any $f\in \bG_i$ for $0 \le i \le j$.
\end{center}
We will prove that the condition $C(k)$ holds.
Consequently, we obtain 
\[
\defset{\In(f)}{f\in \bG_k} = \defset{\In(f)}{f\in G_k}
=\{\In(g_1), \dots, \In(g_t)\}.
\]
Assume that $C(i)$ with $i<k$ holds.
Let $p_1=p_1^{(i)}, p_2=p_2^{(i)}\in G_i$, $\ell=\ell_i$  and $h_j=h_j^{(i)}$ ($j=1, \dots, \ell$).
For any $f(x,y)\in K[x,y]$, we write $f(x,0)=f|_{y=0}$.
Note that $\ord(f)\le N<n$ for each $f\in \cF$. 
Since $\lm(\bar p_1)=\lm(p_1)$ and $\lm(\bar p_2)=\lm(p_2)$ by the assumption, we have $\lm(\spo(\bar p_1, \bar p_2))=\lm(\spo(p_1, p_2))$ and $\spo(\bar p_1, \bar p_2)|_{y=0}=\spo(p_1, p_2)$.
Suppose that $\In(\bh_t)\not\in K[x]$ for some $0\le t\le \ell$ and $t$ is the minimum among such numbers, and let $\Phi(h_t)=\sum_{i=1}^{m+1} a_is_i$. 
Then we have the following representation:  
\[
\bh_t=\sum_{i=1}^m a_if_i +a_{m+1}(g-y^n)=h_t-a_{m+1}y^n.
\]
Since $\ord(h_t) < n$, we have $h_t=0$;  however, it contradicts that $h_j\ne 0$ for every $0\le j \le \ell$ as $h_{\ell}\ne 0$.
Therefore, $\In(\bh_j)\in K[x]$ for $0\le j\le \ell$, and thus $C(i+1)$ holds.
Hence the condition $C(k)$ also holds.

Now, we have to prove that $\bG_k$ is a standard basis of $I_n$. 
To this end, it suffices to show that $\snf(\spo(f,f') \mid \bG_k)=0$ for any $f, f' \in \bG_k$ by \proref{p:snfstd}.
Let $\bp_1, \bp_2 \in \bG_k$ with $p_1, p_2\in G_k$.
We may assume that $\spo(\bar p_1, \bar p_2)\ne 0$.
Let us use the notation $h_j$, $\bh_j$, $0\le j \le \ell$ as above (we extend $\Phi$ to a function on $\cF \cup\{h_0, \dots, h_{\ell}\}$).
Since $G_k$ is a standard basis of $I+\gen{g}$, it follows from \thmref{t:Buchberger} that $h_{\ell}=\nf(\spo(p_1, p_2) \mid G_k)=0$.
By the argument above, we have that $\In(\bh_{j})\in K[x]$ for $0\le j <\ell$. 
Since  the sequence $\{h_0, \dots, h_{\ell}\}$ is obtained by taking $s$-polynomials and $\lm(h_i)=\lm(\bh_i)$ for $0\le j <\ell$, the sequence $\{\bh_0, \dots, \bh_{\ell-1}\}$ satisfies the property in \defref{d:snf}; however, $\bh_{\ell-1}$ might not be an $s$-normal form.
Assume that $h_{\ell}=\spo(h_{\ell-1},h')$ for an element $h'\in G_k\cup\defset{h_i}{i<\ell-1}$, and put $h=\spo(\bh_{\ell-1}, \bh')$.
Let $h=\sum_{i=1}^m a_if_i +a_{m+1}(g-y^n)$, $a_i \in K[x]$,  be a representation obtained naturally from $\bh_{\ell-1}$ and $\bh'$.
Then $h|_{y=0}=h_{\ell}=0$, and thus $h=-a_{m+1}g=\sum_{i=1}^m a_if_i \in \gen{f_1, \dots, f_m}=I$.
Hence $a_{m+1}\in I$, because $I$ is a prime ideal and $g\not\in I$.
Recall that $F:=\{f_1, \dots, f_m\}$ is a standard basis of $I$.  
By \thmref{t:Buchberger}, we have $\nf(a_{m+1}\mid F)=0$.
Since $F \subset  \bG_k$ by the assumption on $\Phi$, it follows that $\snf(a_{m+1} \mid \bG_k)=0$.
By multiplying $y^n$, we obtain an $s$-normal form of $a_{m+1}y^n=-h$, that is, $\snf(h \mid \bG_k)=0$.
Finally, we obtain a sequence as in \defref{d:snf} starting from $\bh_0=\spo(\bar p_1, \bar p_2)$ via $h$, and ending at $0$. That is, $\snf(\spo(\bar p_1, \bar p_2) \mid \bG_k)=0$.
\end{proof}

\begin{rem}
In the proof above, the assumption that $I$ is a prime ideal is only applied to  show that $a_{m+1}g\in I$ implies $a_{m+1}\in I$. 
Therefore, this assumption can be replaced with a weaker condition.
In fact, it suffices to assume that $g$ is  $K[x]/I$-regular, that is,  if $a\in K[x]$ and $ag\in I$, then $a\in I$.
\end{rem}

\begin{ex}
We show some simple examples with $r=2$.
We use the variables $(x,y,z)$ instead of $(x_1, x_2, y)$ and the monomial ordering $\ds$.

(1) 
 Let $I=\gen{xy}\subset K[x,y]$ and $g=x^{\al}$, where $\al$ is a positive integer.
Then $g$ is not $K[x,y]/I$-regular.
Assume that $n> \al \ge 2$.
Then  $\{xy, x^{\al}-z^n, yz^n\}$ is a standard basis of $I_n=\gen{xy, x^{\al}-z^n}$ and $\In (I_n) = \gen{xy, x^{\al}, yz^n}$.
From an exact sequence
\[
0 \to \frac{\In (I_n) +\gen{x}}{\In (I_n)} \to  \frac{K[x,y,z]}{\In (I_n)} \to 
 \frac{K[x,y,z]}{\In (I_n)+\gen{x}} \to 0
\]
and an isomorphism $K[x,y,z]/\gen{y,x^{\al-1}} \cong (\In (I_n) +\gen{x})/\In (I_n)$, we have $\mult(B_n)=\al+n$.

(2) Let $I=\gen{x^3+y^4}$ and $g=x^3$.
Then $\{x^3,y^4\}$ is a standard basis of $I+\gen{g}$. So $N=4$.
Then we have the following:
\renewcommand{\arraystretch}{1.2}
\[
\begin{array}{c|ccccc}
\hline\hline
n & 2 & 3 & 4 & 5 & 6 \\
\hline
\In (I_n) & \gen{z^2,x^3} & \gen{x^3+z^3, z^3} & \gen{x^3,y^4-z^4} & \gen{x^3,y^4} & \gen{x^3,y^4} \\
\hline
\mult(B_n) & 6 & 9 & 12 & 12 & 12  \\
\hline\hline
\end{array}
\]

\end{ex}

\begin{rem}\label{r:TC}
In the situation of \thmref{t:main}, we have 
\[
K[x,y]/\gen{\In(g_1), \dots, \In(g_t)}\cong (K[x]/\gen{\In(g_1), \dots, \In(g_t)})[y].
\]
Hence the affine tangent cone of $B_n$ with $n>N$  is isomorphic to the product of the affine tangent cone of $R/(I+\gen{g})R$ and the affine line $\A^1_K$.
\end{rem}

\begin{rem}\label{r:pg}
If $A$ is normal and the image of $g$ in $A$ is reduced, then $B_n$ are normal for every $n\in \Z_{>0}$ (\cite{tw.Zr-cover}).
Let us consider the case that $\spec B_n$ are normal two-dimensional singularities. Then, as $n$ approaches infinity, the geometric genus $p_g(\spec B_n)$ tends towards infinity (cf. \cite[\S 2]{AshiCyc}).
Therefore, the tangent cone  of a normal two-dimensional singularity cannot provide upper bounds of the geometric genus, even though any singularity is a small deformation of its tangent cone (cf. \cite[\S 5]{tki-w}).
\end{rem}


\providecommand{\bysame}{\leavevmode\hbox to3em{\hrulefill}\thinspace}
\providecommand{\MR}{\relax\ifhmode\unskip\space\fi MR }
\providecommand{\MRhref}[2]{%
  \href{http://www.ams.org/mathscinet-getitem?mr=#1}{#2}
}
\providecommand{\href}[2]{#2}

\end{document}